\documentclass[10pt,reqno]{amsart}
\usepackage{latexsym,amssymb,amsmath}
\usepackage{amsmath,amsfonts}
\usepackage{epsfig}
\usepackage{graphicx}
\usepackage{verbatim,xcolor}
\usepackage{dirtytalk}
\usepackage{footnote}
\usepackage[pagewise]{lineno}

\newtheorem{theorem}[subsection]{Theorem}
\newtheorem{lemma}[subsection]{Lemma}

\newtheorem{proposition}[subsection]{Proposition}
\usepackage{hyperref}
\usepackage{amsmath}
\usepackage{amssymb}
\usepackage{mathtools}
\usepackage{scalerel,stackengine}

\usepackage[T1]{fontenc}

\stackMath
\newcommand\reallywidehat[1]{%
\savestack{\tmpbox}{\stretchto{%
  \scaleto{%
    \scalerel*[\widthof{\ensuremath{#1}}]{\kern-.6pt\bigwedge\kern-.6pt}%
    {\rule[-\textheight/2]{1ex}{\textheight}}
  }{\textheight}%
}{0.5ex}}%
\stackon[1pt]{#1}{\tmpbox}%
}
\allowdisplaybreaks

\newcommand{\Hmm}[1]{\leavevmode{\marginpar{\tiny%
$\hbox to 0mm{\hspace*{-0.5mm}$\leftarrow$\hss}%
\vcenter{\vrule depth 0.1mm height 0.1mm width \the\marginparwidth}%
\hbox to 0mm{\hss$\rightarrow$\hspace*{-0.5mm}}$\\\relax\raggedright #1}}}

\numberwithin{equation}{section}

\begin{document}

\title[Global Dynamics of the Biharmonic NLS]{GLOBAL WELL-POSEDNESS FOR THE BIHARMONIC QUINTIC NONLINEAR SCHRÖDINGER EQUATION ON $\mathbb{R}^2$}

\author[Ba\c{s}ako\u{g}lu, G\"{u}rel, Yılmaz] {Engin Ba\c{s}ako\u{g}lu, T. Burak G\"{u}rel, O\u{g}uz Yılmaz}

\address{Department of Mathematics,
Bo\u gazi\c ci University, 
Bebek 34342, Istanbul, Turkey}
\email{engin.basakoglu@boun.edu.tr}
\address{Department of Mathematics,
Bo\u gazi\c ci University, 
Bebek 34342, Istanbul, Turkey}
\email{bgurel@boun.edu.tr}
\address{Department of Mathematics, Bo\u gazi\c ci University, 
Bebek 34342, Istanbul, Turkey}
\email{oguz.yilmaz@boun.edu.tr}

\begin{abstract}
We prove that the Cauchy problem for the 2D quintic defocusing biharmonic Schrödinger equation is globally well-posed in the Sobolev spaces $H^s(\mathbb{R}^2)$ for $\frac{8}{7}<s<2$. Our main ingredient to establish the result is the $I$-method of Colliander-Keel-Staffilani-Takaoka-Tao \cite{colliander2002almost} which is used to construct the modified energy functional that is almost conserved in time.  

\end{abstract}

\maketitle
\section{Introduction}
In this paper, we consider the initial value problem (IVP) for the quintic defocusing biharmonic Schrödinger equation
\begin{equation}\label{eq:4NLS}
    \begin{cases}
    &i\partial_{t}u-\triangle^{2}u=\vert u\vert^{4}u,\quad (t,x)\in\mathbb{R}\times\mathbb{R}^2,\\
    &u(0,x)=u_{0}(x)\in H^{s}(\mathbb{R}^{2}),
    \end{cases}
\end{equation} 
where $u$ is a complex-valued space-time function. The fourth-order nonlinear Schrödinger equations (4-NLS)
\begin{equation*}\label{eq:karpman_4NLS}
    i\partial_{t}\psi+\frac{1}{2}\triangle\psi+\frac{\gamma}{2}\triangle^{2}\psi+f(\vert\psi\vert^{2})\psi=0,\quad f(u)=u^{p},\,p\geq 1,\,\gamma\in\mathbb{R}
\end{equation*}
were introduced by Karpman and Shagalov, \cite{karpman1996stabilization,karpman2000stability}, to take into account the effect of small fourth-order dispersion terms in the propagation of intense laser beams in a bulk medium with Kerr nonlinearity. 
The local solutions to \eqref{eq:4NLS} obey mass and energy conservation laws. More precisely, 
\begin{equation*}
    M(u(t))=\Vert u(t)\Vert_{L^{2}(\mathbb{R}^{2})}=M(u_{0})
\end{equation*}
and
\begin{equation}\label{energyfunc}
    E(u(t))=\frac{1}{2}\int_{\mathbb{R}^{2}}\vert\triangle u(t,x)\vert^{2}+\frac{1}{3}\vert u(t,x)\vert^{6}\,\text{d}x=E(u_{0}).
\end{equation}
If $u(t,x)$ is a solution of \eqref{eq:4NLS} with initial data $u_{0}$, then for $\lambda>0$,
\begin{equation}\label{eq:scalesymm}
    u_{\lambda}(t,x):=\lambda^{-1}u\left(\frac{t}{\lambda^{4}},\frac{x}{\lambda}\right)
\end{equation}
solves the equation \eqref{eq:4NLS} with initial data $u_{0,\lambda}(x)=\lambda^{-1}u_{0}(x)$. Using the scaled solution \eqref{eq:scalesymm}, associated to \eqref{eq:4NLS} posed on $\mathbb{R}^d$, the $\Dot{H^{s}}(\mathbb{R}^{d})$ criticality index $s_c$ can be determined as 
\begin{equation*}
    s_{c}=\frac{d-2}{2}.
\end{equation*}
Therefore, in our case we have 
 $s_c=0$ meaning that the equation \eqref{eq:4NLS} is mass critical. To review the literature, we shall write the general form of 4-NLS 
\begin{equation}\label{eq:G4NLS}
    i\partial_{t}u+\alpha\triangle u+\beta\triangle^{2}u+N(u)=0,
\end{equation}
where $\alpha,\beta\in\mathbb{R}$ and $N(u)$ is the nonlinear term. Davydova and Zaliznyak \cite{davydova2001schrodinger} considered \eqref{eq:G4NLS} for nonlinearities of the form
\begin{equation}\label{eq:cub-quin_nonl}
    N(u)=\gamma \vert u\vert^{2}u+\theta\vert u\vert^{4}u
\end{equation}
with $\gamma\theta<0$ (saturable nonlinearity) and certain restrictions on $\alpha,\beta$. Using analytical and numerical methods, they investigated the spatial-temporal wave packet dynamics in the vicinity of the stationary (soliton) solution. In particular, by means of a variational approach, the prediction of the existence conditions and the stability properties of the chirped and ordinary solitons to the equation \eqref{eq:G4NLS} were given. Considering the physical reality, the cubic-quintic type of nonlinearities \eqref{eq:cub-quin_nonl} are known to give rise to a formation of sufficiently robust optical vortices for 2D cylindrical light beam propagation, \cite{Quiroga-Teixeiro:97}. Also the authors notify in \cite{davydova2001schrodinger} that the existence of the quintic correction term in \eqref{eq:cub-quin_nonl} leads to the most universal approximation for saturable nonlinearities because any model regarding small nonlinearity could be reduced to the quintic model. In \cite{pausader_mass_crit_2010}, Pausader and Shao pointed out that when $\vert\beta\vert$ is sufficiently large with respect to $\vert\alpha\vert$, the waveguides corresponding to the NLS part of \eqref{eq:G4NLS} become stable which in turn yields that the equation \eqref{eq:G4NLS} is predominantly governed by the corresponding biharmonic equation
\begin{equation}\label{eq:G_biharmonic_NLS}
    i\partial_{t}u+\beta\triangle^{2}u+N(u)=0.
\end{equation}
Appropriately changing the time variable, one can take $\beta=\pm 1$. It is remarkable at this point to note that the motivation of taking the biharmonic model with quintic nonlinearity into account comes from the aforementioned explanations and the fact that there is no known result for the global well-posedness of the biharmonic NLS below the energy level in spatial dimension $d=2$ . Fourth-order Schrödinger equations with mixed dispersion terms as in \eqref{eq:G4NLS} will be addressed in a future study. 

As far as the physical point of view is concerned, the biharmonic equation \eqref{eq:G_biharmonic_NLS} is very well studied in deep water wave dynamics \cite{Dysthe}, three dimensional motion of an isolated vortex filament embedded in inviscid incompressible fluid filling in an unbounded region \cite{segata}, and solitary waves \cite{karpman1996stabilization,karpman2000stability}. Moreover, the biharmonic NLS was considered in \cite{Turitsyn1985} as a sample model to study the stability of solitons in magnetic materials when the impact of quasiparticle mass becomes arbitrarily large. Well-posedness of the IVP of \eqref{eq:G4NLS} with various nonlinearities have been widely studied in the literature. Next we review some of these results that are the most relevant to our work. In \cite{Pausader2007GlobalWF}, Pausader established the global well-posedness and discussed the scattering of the solution of the IVP of \eqref{eq:G4NLS} with $\alpha\leq 0$, $\beta=1$ and radially symmetric data $u_{0}\in H^{2}(\mathbb{R}^{d})$, $d\geq 5$, along with the nonlinearities $N(u)=\vert u\vert^{p}u$, $p\in(0,\frac{8}{d-4}]$. In particular, when $p\neq \frac{8}{d-4}$, the radial symmetry assumption on the data could be removed and the global well-posedness was shown to hold still true in $H^{2}(\mathbb{R}^{d})$ without radial assumption on the data. Later, Miao, Xu and Zhao \cite{miao2011global} improved the result in \cite{Pausader2007GlobalWF} for a particular case in which $\alpha=0$, $\beta=1$ and $d\geq 9$ with the nonlinearity $N(u)=\vert u\vert^{\frac{8}{d-4}}u$. Wang \cite{wang_2012} obtained the local and global well-posedness of the IVP for the equation \eqref{eq:G_biharmonic_NLS} on $\mathbb{R}\times\mathbb{R}$ with $\beta=1$ and $N(u)=\partial_{x}(\vert u\vert^{2k}u)$, $k\geq 2$, for small initial data $u_{0}\in \Dot{H}^{s_{k}}(\mathbb{R})$, $s_{k}=\frac{k-3}{2k}$. The IVP for the cubic biharmonic NLS \eqref{eq:G_biharmonic_NLS} posed on $\mathbb{R}\times\mathbb{R}^{d}$ with $\beta=1$ was discussed by Pausader in \cite{pausader2009cubic}. For $1\leq d\leq 8$ and any initial data in $H^{2}(\mathbb{R}^{d})$, he proved the existence of a unique global solution and the analyticity of data-to-solution map. Moreover, for spatial dimensions $5\leq d\leq 8$, it was shown that any $H^{2}$ solution scatters in time. Also, if $d\geq 9$, the equation was proved to be ill-posed in $H^{2}$ by showing the existence of a Schwartz function $u_{0}$ and a solution $u\in C([0,\varepsilon],H^{2})$ for any $\varepsilon>0$ with data $u_{0}$ satisfying $\Vert u_{0}\Vert_{H^{2}}<\varepsilon$ but $\Vert u(t_{\varepsilon})\Vert_{H^{2}}>\varepsilon^{-1}$, $t_{\varepsilon}\in(0,\varepsilon)$. The global well-posedness of the IVP for \eqref{eq:G4NLS} on $\mathbb{R}\times\mathbb{R}^{d}$ with either $\alpha<0$, $\beta>0$ and $N(u)=\vert u\vert^{2m}u$, $m\in\mathbb{Z}^{+}$ or $\alpha>0$, $\beta<0$ and $N(u)=-\vert u\vert^{2m}u$, $m\in\mathbb{Z}^{+}$, where certain restrictions are employed on $m$ and $d$, was obtained by Guo in \cite{Guo2010}. In particular, using the $I$-method of Colliander-Keel-Staffilani-Takaoka-Tao \cite{colliander2002almost}, the equation \eqref{eq:G4NLS} with restrictions specified above was proved to be globally well-posed in $H^{s}(\mathbb{R}^{d})$ for $$s=1+\frac{md-9+\sqrt{(4m-md+7)^{2}+16}}{4m}$$
where $4<md<4m+2$. The IVP for \eqref{eq:G_biharmonic_NLS} with $\beta=1$ and the cubic nonlinearity on $\mathbb{R}\times\mathbb{R}^{d}$, $5\leq d\leq 7$, was addressed by Miao, Wu and Zhang in \cite{miao_2015} for which the related global well-posedness and scattering results were established in $H^{s}(\mathbb{R}^{d})$ for $s>\frac{16(d-4)}{7d-24}$ if $d=5,6$; also for $s>\frac{45}{23}$ if $d=7$. The idea of their approach is the application of the $I$-method combined with an interaction Morawetz type estimate which is available for $d\geq 5$. In \cite{dinh2017global}, Dinh obtained the global well-posedness for the equation \eqref{eq:G_biharmonic_NLS} with $\beta=1$ and $N(u)=\vert u\vert^{p}u$, $p\in(\frac{8}{d},\frac{8}{d-4})$ below the energy space by using the $I$-method together with the interaction Morawetz inequality. In \cite{seong2021well}, the well-posedness and the ill-posedness for the IVP of \eqref{eq:G_biharmonic_NLS} posed on $\mathbb{R}\times\mathbb{R}$ with $\beta=1$ and $N(u)=\pm\vert u\vert^{2}u$ were studied by Seong. It was obtained that the equation is globally well-posed in $H^{s}(\mathbb{R})$ when $s\geq-\frac{1}{2}$ whereas it is ill-posed for $s<-\frac{1}{2}$ since the corresponding data-to-solution map ceases to be uniformly continuous in this regime. In the case of the initial-boundary value problem (IBVP), the global well-posedness for the biharmonic Schr\"{o}dinger equation \eqref{eq:G_biharmonic_NLS} (with $\beta=1$) with inhomogeneous Dirichlet-Neumann boundary data  has recently been established in the energy space $H^2(\mathbb{R}^+)$ up to cubic nonlinearities by Özsarı and Yolcu, \cite{turker2019}. The challenge for constructing global solution for the biharmonic equation posed on the half-line lies in the fact that unlike the IVP, the IBVP for the biharmonic equation does not satisfy the energy conservation laws once the boundary data are nonzero. Later, Başakoğlu \cite{Basakoglu2021} has extended the result of \cite{turker2019} for the cubic biharmonic equation to more regular spaces  by using the nonlinear smoothing property of the equation and has obtained the linear growth bound for the solution on the half-line. 

 Xia and Pausader \cite{pausader2013scattering} established scattering in the energy space $H^{2}$ for low spatial dimensions $1\leq d\leq4 $. The difficulty of obtaining global well-posedness and scattering result in low dimensions is the lack of Morawetz-type inequalities adapted to small dimensions. To overcome this difficulty the authors construct a new virial-type estimate to establish the scattering result. In this work, we aim to obtain the global well-posedness for the IVP \eqref{eq:4NLS} below the energy space. In this regard, our approach, which is inspired by \cite{colliander2002almost}, in the most broad terms will be to generate an almost conserved quantity by modifying the energy functional \eqref{energyfunc} and then using it to find a priori polynomial-in-time bound for the $H^{s}$-norm of the solution. Our main result is as follows:
\begin{theorem}\label{main_thm}
The initial value problem \eqref{eq:4NLS} is globally well-posed for initial data $u_{0}\in H^{s}(\mathbb{R}^{2})$ for $s>\frac{8}{7}$.
\end{theorem}
It is worthwhile to note that 4-NLS \eqref{eq:G4NLS} does not possess the scaling symmetry when $\alpha\neq 0$; on the contrary the biharmonic equation \eqref{eq:G_biharmonic_NLS} has the scaling invariance with respect to \eqref{eq:scalesymm} which will be the essential property in obtaining polynomial-in-time bound for the $H^{s}$-norm of the solution.

The organization of the paper is as follows. In Section 2, we introduce required notation, function spaces and a priori estimates. In Section 3, using the growth bound for the modified energy functional, we prove the main result. In the final section, we give a proof of the almost conservation law.
\section{Background and Notation}
We write $X\lesssim Y$ if there exists a positive constant $K>2$ independent of $X,Y$ such that $X\leq KY$, also write $X\sim Y$ if $X\lesssim Y$ and $Y\lesssim X$. We denote $X\ll Y$ when $Y> KX$. Let $\langle X\rangle= \sqrt{1+X^{2}}$, and $\langle\nabla\rangle$ denote the operator with Fourier multiplier $\langle\xi\rangle$ via the Fourier transform 
\begin{equation*}
    \widehat{u}(\xi)=\int_{\mathbb{R}^{d}}e^{-ix\cdot\xi}u(x)\,\text{d}x.
\end{equation*}
 Similarly the space time Fourier transform is defined as\begin{equation*}\widetilde{u}(\tau,\xi)=\int_{\mathbb{R}^{d+1}}e^{-ix\cdot\xi-it\tau}u(t,x)\,\text{d}x\,\text{d}t.\end{equation*} Let $s\in\mathbb{R}$. We define the homogeneous Sobolev spaces \begin{equation*}
 \Dot{H}^{s}(\mathbb{R}^{d})=\{f\in\mathcal{S}'(\mathbb{R}^d): \Vert u\Vert_{\dot{H}^{s}(\mathbb{R}^{d})}=\Vert|\xi|^{s}\widehat{u}(\xi)\Vert_{L^{2}_{\xi}(\mathbb{R}^{d})}<\infty\}
 \end{equation*}
 and the nonhomogeneous Sobolev spaces $H^s(\mathbb{R}^d)$ are defined analogously with $\langle \cdot\rangle$ instead of $|\cdot|$. The Fourier restriction space $X^{s,b}$ associated with the equation \eqref{eq:4NLS} is defined to be the closure of the Schwartz functions $\mathcal{S}_{t,x}(\mathbb{R}\times \mathbb{R}^d)$ under the norm
\begin{equation}\label{eq:X_sb}
    \Vert u\Vert_{X^{s,b}(\mathbb{R}\times \mathbb{R}^d)}=\Vert\langle\xi\rangle^{s}\langle\tau+\vert\xi\vert^{4}\rangle^{b}\widetilde{u}(\tau,\xi)\Vert_{L^{2}_{\tau,\xi}(\mathbb{R}\times\mathbb{R}^d)},
\end{equation}
and the corresponding Fourier restricted norm is given as
\begin{equation}\label{eq:trunc-X_sb}
    \Vert u\Vert_{X^{s,b}_{\delta}}=\inf_{v=u\text{ on }[0,\delta]}\Vert v\Vert_{X^{s,b}}.
\end{equation}
We shall write $\frac{1}{2}\pm\equiv\frac{1}{2}\pm\varepsilon$ for some universal $0<\varepsilon\ll 1$.
Let $\varphi$ be a real-valued, smooth, compactly supported, radial function such that $\text{supp}\,\varphi\subseteq\{\xi\in\mathbb{R}^{n}:\vert\xi\vert\leq 2\}$ and $\varphi\equiv 1$ on the closed unit ball. The Littlewood-Paley projection operators are defined by
\begin{align}
    \widehat{P_{N}u}(\xi)&=(\varphi(\xi/N)-\varphi(2\xi/N))\widehat{u}(\xi)\label{eq:L-P_opt}\\ \nonumber\widehat{P_{\leq N}u}(\xi)&=\sum_{M\leq N}\widehat{P_{M}u}(\xi)=\varphi(\xi/N)\widehat{u}(\xi),\\\nonumber \widehat{P_{>N}u}(\xi)&=\sum_{M>N}\widehat{P_{M}u}(\xi)=(1-\varphi(\xi/N))\widehat{u}(\xi)
\end{align}
where $M,N$ are dyadic numbers and the sums are taken over $2^{j}$, $j\in\mathbb{Z}$. In our discussion the following Bernstein's estimate is useful: for $s\geq 0$ and $1\leq p\leq\infty$,
\begin{equation}\label{eq:Berns_ineq}
    \Vert P_{N}u\Vert_{L^{p}_{x}(\mathbb{R}^{n})}\sim_{p,n}N^{\pm s}\Vert\vert\nabla\vert^{\mp s}P_{N}u\Vert_{L^{p}_{x}(\mathbb{R}^{n})}.
\end{equation}
Also, we will use the Littlewood-Paley estimate 
\begin{equation}\label{eq:L-P_ineq}
    \Vert u\Vert_{L^{p}_{x}(\mathbb{R}^{n})}\sim_{p,n}\Vert\big(\sum_{N}\vert P_{N}u\vert^{2}\big)^{1/2}\Vert_{L^{p}_{x}(\mathbb{R}^{n})}
\end{equation}
for $1<p<\infty$. Combining \eqref{eq:Berns_ineq}, \eqref{eq:L-P_ineq} and Plancherel's theorem, we get
\begin{equation*}
    \Vert u\Vert_{\Dot{H}^{s}_{x}(\mathbb{R}^{n})}\sim_{s,n}\Big(\sum_{N}N^{2s}\Vert P_{N}u\Vert_{L^{2}_{x}(\mathbb{R}^{n})}\Big)^{1/2}
\end{equation*}
and
\begin{equation*}
    \Vert u\Vert_{H^{s}_{x}(\mathbb{R}^{n})}\sim_{s,n}\Vert P_{\leq 1}u\Vert_{L^{2}_{x}(\mathbb{R}^{n})}+\Big(\sum_{N>1}N^{2s}\Vert P_{N}u\Vert_{L^{2}_{x}(\mathbb{R}^{n})}\Big)^{1/2},
\end{equation*}
see for instance \cite{tao2006nonlinear}. Lastly, we define the $L^{q}_{t}L^{p}_x$ norm by
\begin{equation*}\label{eq:LpLqnorm}
    \Vert f\Vert_{L^{q}_{t}L^{p}_{x}(\mathbb{R}\times\mathbb{R}^{d})}=\left(\int_{\mathbb{R}}\left(\int_{\mathbb{R}^{d}}\vert f(t,x)\vert^{p}\,\text{d}x\right)^{q/p}\text{d}t\right)^{1/q}
\end{equation*}
and the pair $(p,q)$ is said to be biharmonic admissible if \begin{equation*} 
    \frac{4}{q}=d\Big(\frac{1}{2}-\frac{1}{p}\Big)\,\,\text{and}\,\, \begin{cases}
2\leq p < \frac{2d}{d-4} \hspace{0.5cm}\text{if}\,\,d\geq 4\\ 2\leq p \leq \infty \hspace{0.75cm}\text{if}\,\,d<4.
\end{cases}
\end{equation*}
Note that $(p,q,d)\neq (\infty,2,4)$.
Below we state the Strichartz estimates.
\begin{theorem}[\cite{cui2007well}]\label{str_est}
    Assume that $0\leq\mu\leq 1$, $\frac{2}{1-\mu}\leq p\leq\infty$, $2\leq q\leq\infty$, and
    \begin{equation}\label{bihadmissible}
        \frac{4}{q}=2\Big(\frac{1}{2}-\frac{1}{p}\Big)+\mu.
    \end{equation}
    Then for any $T_{0}>0$, there exists a constant $C>0$ depending on $p,q,T_{0},\mu$ such that for any $0<T<T_{0}$ and $\varphi\in L^{2}(\mathbb{R}^{2})$, we have
    \begin{align}
\Vert\vert\nabla\vert^{\mu}e^{it\triangle^{2}}\varphi\Vert_{L^{q}_{t}L^{p}_{x}([0,T]\times\mathbb{R}^{2})}&\leq C\Vert\varphi\Vert_{L^{2}_{x}(\mathbb{R}^{2})},\label{eq:str_est_hom}\\
\Vert\langle\nabla\rangle^{\mu}e^{it\triangle^{2}}\varphi\Vert_{L^{q}_{t}L^{p}_{x}([0,T]\times\mathbb{R}^{2})}&\leq C\Vert\varphi\Vert_{L^{2}_{x}(\mathbb{R}^{2})}\label{eq:str_est_inhom}.
    \end{align}
    In particular, for any biharmonic admissible pair $(p,q)$, i.e. if $\mu=0$ in \eqref{bihadmissible}, we have
    \begin{equation}
        \Vert e^{it\triangle^{2}}\varphi\Vert_{L^{q}_{t}L^{p}_{x}([0,T]\times\mathbb{R}^{2})}\leq C\Vert\varphi\Vert_{L^{2}_{x}(\mathbb{R}^{2})}.\label{eq:str_est_adm}
    \end{equation}
\end{theorem}
\noindent
\textbf{Remark.} In our case, we cannot efficiently use a kind of bilinear Strichartz estimate of Lemma 2.1 of \cite{colliander2002almost} (see for instance \cite{seong2021well} for the fourth-order version of the bilinear Strichartz estimate) especially when the magnitude of the interacting frequencies are comparable. To be able to handle the multilinear terms formed by functions with comparable magnitude of frequencies, one can utilize bilinear estimates only if the spatial dimension $n$ is at least $4$. Indeed, let $u_{1},u_{2}\in X^{0,b}([0,T]\times\mathbb{R}^{d})$, $b>1/2$, be two free solutions of the biharmonic NLS equation with initial data $f_1$, $f_2$ respectively. Assume that the Fourier supports of $u_1$ and $u_2$ lie within $\{\xi\in\mathbb{R}^{d}:\vert\xi\vert\sim N_{1}\}$ and $\{\xi\in\mathbb{R}^{d}:\vert\xi\vert\sim N_{2}\}$ respectively. If $N_{1}\sim N_{2}$, then we cannot follow a change of variable argument as in the proof of Lemma 3.4 of \cite{colliander2008global}. Therefore, we are limited to the Sobolev embedding and the Strichartz estimate 
\eqref{eq:str_est_adm} to show an inequality of the following type
\begin{equation*}
    \Vert u_{1}u_{2}\Vert_{L^{2}_{t,x}}\lesssim\Vert f_{1}\Vert_{\Dot{H}^{a_{1}}_{x}}\Vert f_{2}\Vert_{\Dot{H}^{a_{2}}_{x}}
\end{equation*}
where
\begin{equation*}
    a_{1}=a_{2}=\frac{d}{4}-1.
\end{equation*}
To see this, first apply the Hölder's inequality to get
\begin{equation*}
     \Vert u_{1}u_{2}\Vert_{L^{2}_{t,x}}\leq\Vert e^{it\triangle^{2}}f_{1}\Vert_{L^{4}_{t,x}}\Vert e^{it\triangle^{2}}f_{2}\Vert_{L^{4}_{t,x}}.
\end{equation*}
In dimension $d$, the pair $(\frac{2d}{d-2}, 4)$ is biharmonic admissible. Thus, applying the Sobolev embedding and then the Strichartz estimate \eqref{eq:str_est_adm} to the factors of the above bound yields that
\begin{equation*}
    \Vert e^{it\triangle^{2}}f_{j}\Vert_{L^{4}_{t,x}}\lesssim\Vert \vert\nabla\vert^{\frac{d-4}{4}}e^{it\triangle^{2}}f_{j}\Vert_{L^{4}_{t}L^{\frac{2d}{d-2}}_{x}}\lesssim\Vert f_{j}\Vert_{\Dot{H}^{\frac{d-4}{4}}_{x}}\quad j=1,2.
\end{equation*}
Note that the Sobolev embedding in the first inequality above is valid provided that $d\geq 4$ which determines whether or not we are able to use bilinear type estimates for pair of functions with comparable frequencies.

In our situation, combining Theorem \ref{str_est} and Lemma 2.9 from \cite{tao2006nonlinear}, we obtain $X^{s,b}$ type Strichartz estimates with derivative gain:
\begin{lemma}\label{X_s,b-str}
    Let $u\in X^{0,\frac{1}{2}+}$ whose spatial frequency is supported in the set $\{\xi\in\mathbb{R}^{2}:\vert\xi\vert\sim N\}$ for some $N\geq 1$. Then for sufficiently small $\delta>0$, we have
    \begin{equation}
        \Vert\vert\nabla\vert^{\mu} u\Vert_{L^{q}_{t}L^{p}_{x}([0,\delta]\times\mathbb{R}^{2})}\lesssim\Vert u\Vert_{X_{\delta}^{0,\frac{1}{2}+}}\label{eq:X_s,b-str}
    \end{equation}
    where $\mu,p,q$ are given as in Theorem \ref{str_est}.
\end{lemma}
\begin{proof}
    By Fourier inversion, we can write
    \begin{equation*}
u(t,x)=\int_{\mathbb{R}}\int_{\mathbb{R}^{2}}\widetilde{u}(\tau,\xi)e^{it\tau}e^{ix\cdot\xi}\,\text{d}\xi \text{d}\tau.
    \end{equation*}
     Set $\tau=\tau_{0}-\vert\xi\vert^{4}$ and define
    \begin{equation*}
        f_{\tau_{0}}(x)=\int_{\mathbb{R}^{2}}\widetilde{u}(\tau_{0}-\vert\xi\vert^{4},\xi)e^{ix\cdot\xi}\,\text{d}\xi.
    \end{equation*}
By noting that
    \begin{equation*}
        \widehat{f_{\tau_{0}}}(\xi)=\widetilde{u}(\tau_{0}-\vert\xi\vert^{4},\xi)
    \end{equation*}
    and 
    \begin{equation*}
        e^{it\triangle^{2}}f_{\tau_{0}}(x)=\int_{\mathbb{R}^{2}}\widetilde{u}(\tau_{0}-\vert\xi\vert^{4},\xi)e^{it\vert\xi\vert^{4}}e^{ix\cdot\xi}\,\text{d}\xi,
    \end{equation*}
we may write
    \begin{equation*}
        u(t,x)=\int_{\mathbb{R}}e^{it\tau_{0}}e^{it\triangle^{2}}f_{\tau_{0}}(x)\,\text{d}\tau_{0}.
    \end{equation*}
    Now, to implement Lemma 2.9 from \cite{tao2006nonlinear}, we need to show that $f_{\tau_{0}}\in L^{2}(\mathbb{R}^{2})$ for all $\tau_{0}\in\mathbb{R}$. Indeed, it is enough to exhibit $\widehat{f_{\tau_{0}}}\in L^{2}(\mathbb{R}^{2})$:
    \begin{align*}
        \Vert \widehat{f_{\tau_{0}}}\Vert_{L^{2}}^{2}&=\int\vert\widetilde{u}(\tau_{0}-\vert\xi\vert^{4},\xi)\vert^{2}\,\text{d}\xi\\
        &=\int\Bigg\vert\int\widetilde{u}(\tau,\xi)\delta(\tau+\vert\xi\vert^{4}-\tau_{0})\,\text{d}\tau\Bigg\vert^{2} \,\text{d}\xi\\
        &=\int\Bigg\vert\int\widetilde{u}(\tau,\xi)\delta(\tau+\vert\xi\vert^{4}-\tau_{0})\langle\tau+\vert\xi\vert^{4}\rangle^{\frac{1}{2}+}\langle\tau+\vert\xi\vert^{4}\rangle^{-\frac{1}{2}-} \,\text{d}\tau\Bigg\vert^{2} \text{d}\xi\\
        &\leq\int\Bigg(\int\vert\widetilde{u}(\tau,\xi)\vert\delta(\tau+\vert\xi\vert^{4}-\tau_{0})\langle\tau+\vert\xi\vert^{4}\rangle^{\frac{1}{2}+}\langle\tau+\vert\xi\vert^{4}\rangle^{-\frac{1}{2}-}\,\text{d}\tau\Bigg)^{2} \,\text{d}\xi\\
        &\leq\int\Bigg(\int\vert\widetilde{u}(\tau,\xi)\vert^{2}\langle\tau+\vert\xi\vert^{4}\rangle^{1+}\,\text{d}\tau\Bigg)\Bigg(\int \delta(\tau+\vert\xi\vert^{4}-\tau_{0})\langle\tau+\vert\xi\vert^{4}\rangle^{-1-}\,\text{d}\tau\Bigg)d\xi\\
        &\sim\Vert u\Vert_{X^{0,\frac{1}{2}+}}^{2}\sup_{\{\vert\xi\vert\sim N,\tau+\vert\xi\vert^{4}=\tau_{0}\}}\int \langle\tau_{0}\rangle^{-1-}\,\text{d}\tau_{0}\lesssim \Vert u\Vert_{X^{0,\frac{1}{2}+}}^{2}.
    \end{align*}
    So, we conclude that $f\in L^{2}(\mathbb{R}^{2})$. Also, by Theorem \ref{str_est}, we have
    \begin{equation*}
        \Vert e^{it\tau_{0}}e^{it\triangle^{2}}\vert\nabla\vert^{\mu}f_{\tau_{0}}\Vert_{L^{q}_{t}L^{p}_{x}([0,\delta]\times\mathbb{R}^{2})}\lesssim\Vert f_{\tau_{0}}\Vert_{L^{2}_{x}(\mathbb{R}^{2})}.
    \end{equation*}
    Then, utilizing the above estimate and returning back the original variable $\tau$ we arrive at
    \begin{align*}
        \Vert \vert\nabla\vert^{\mu}u\Vert_{L^{q}_{t}L^{p}_{x}([0,\delta]\times\mathbb{R}^{2})}&=\Bigg\Vert\int_{\mathbb{R}}e^{it\tau_{0}}e^{it\triangle^{2}}\vert\nabla\vert^{\mu}f_{\tau_{0}}(x)\,\text{d}\tau_{0}\Bigg\Vert_{L^{q}_{t}L^{p}_{x}([0,\delta]\times\mathbb{R}^{2})}\\
        &\leq\int_{\mathbb{R}}\Vert e^{it\tau_{0}}e^{it\triangle^{2}}\vert\nabla\vert^{\mu}f_{\tau_{0}}(x)\Vert_{L^{q}_{t}L^{p}_{x}([0,\delta]\times\mathbb{R}^{2})}\,\text{d}\tau_{0}\\
        &\lesssim\int_{\mathbb{R}}\Vert f_{\tau_{0}}\Vert_{L^{2}_{x}(\mathbb{R}^{2})}\,\text{d}\tau_{0}\\
        &\leq\Bigg(\int_{\mathbb{R}}\Vert\widehat{f_{\tau_{0}}}\Vert^{2}_{L^{2}_{\xi}(\mathbb{R}^{2})}\langle\tau_{0}\rangle^{1+}\text{d}\tau_{0}\Bigg)^{1/2}\Bigg(\int_{\mathbb{R}}\langle\tau_{0}\rangle^{-1-}\,\text{d}\tau_{0}\Bigg)^{1/2}\\
        &\lesssim\Bigg(\int\int\langle\tau+\vert\xi\vert^{4}\rangle^{1+}\vert\widetilde{u}(\tau,\xi)\vert^{2}\,\text{d}\xi \text{d}\tau\Bigg)^{1/2}=\Vert u\Vert_{X^{0,\frac{1}{2}+}_{\delta}}.
    \end{align*}
\end{proof} \noindent
\textbf{Remark.} In view of \eqref{eq:X_s,b-str} and \eqref{eq:Berns_ineq}, we have
\begin{equation}\label{keyestimate}
    \Vert u\Vert_{L^{q}_{t}L^{p}_{x}([0,\delta]\times\mathbb{R}^{2})}\lesssim N^{-\mu}\Vert u\Vert_{X^{0,\frac{1}{2}+}_{\delta}}.
\end{equation}
\section{Almost Conservation and Proof of the Main Theorem}
Given $s<2$ and a parameter $N\gg 1$, define the Fourier multiplier operator 
\begin{equation}\label{eq:I-opt}
    \widehat{I_{N}f}(\xi)=m_{N}(\xi)\widehat{f}(\xi),
\end{equation}
where 
\begin{equation}\label{eq:defn-of-m}
    m_{N}(\xi)=\begin{cases}
    1\quad&\vert\xi\vert\leq N,\\
    \vert\xi\vert^{s-2}N^{2-s}\quad &\vert\xi\vert>2N.
    \end{cases}
\end{equation}
that is smooth, radial, nonincreasing in $\vert\xi\vert$.
For simplicity, we shall drop the subscript $N$ in \eqref{eq:I-opt} and \eqref{eq:defn-of-m}. The multiplier $m$ satisfies the condition
\begin{equation*}
    \vert\nabla_{\xi}^{j}m\vert\lesssim\vert\xi\vert^{-j}\text{ for $j\geq 0$, where $\xi\in\mathbb{R}^{n}\setminus\{0\}$}
\end{equation*}
implying that $m$ is a Hörmander-Mikhlin multiplier \cite{Shamir1966}. Consequently, the operator $I$ is bounded on $L^{p}(\mathbb{R}^{n})$ for $1<p<\infty$. Note that
\begin{align}
    E(Iu(t))\leq&\left(N^{2-s}\Vert u(t)\Vert_{\Dot{H}^{s}_{x}(R^{n})}\right)^{2}+\Vert u(t)\Vert_{L^{6}_{x}(\mathbb{R}^{n})}^{6},\label{eq:me<Hs}\\
    \Vert u(t)\Vert_{H^{s}_{x}(\mathbb{R}^{n})}^{2}&\lesssim E(Iu(t))+\Vert u_{0}\Vert_{L^{2}_{x}(\mathbb{R}^{n})}^{2}.\label{eq:Hs<me}
\end{align}
In order to establish Theorem \ref{main_thm}, by the usual density argument, it suffices to show that the solution of \eqref{eq:4NLS} with a compactly supported smooth initial data grows at most polynomially in the $H^{s}_{x}$ norm:
\begin{equation}\label{eq:polbound}
    \Vert u(t)\Vert_{H^{s}_{x}(\mathbb{R}^{2})}\lesssim C_{1}t^{M}+C_{2},
\end{equation}
where the constants $C_{1},C_{2},M$ depend on $H^s$ norm of initial data. The next proposition will play a crucial role in establishing \eqref{eq:polbound}. The main idea is to control the growth of the almost conserved quantity \begin{equation*}
    E(Iu(t))=\int_{\mathbb{R}^{2}}\frac{1}{2}\vert\triangle Iu(t,x)\vert^{2}+\frac{1}{6}\vert Iu(t,x)\vert^{6}\,\text{d}x
\end{equation*}
by means of powers of $N$ depending on the spatial dimension. Time differentiation gives that
\begin{align*}
    \partial_{t}E(Iu(t))&=\Re \int_{\mathbb{R}^{2}}\triangle^{2}Iu\partial_{t}\overline{Iu}+|Iu|^4Iu\,\partial_{t}\overline{Iu}\,\text{d}x\\&=\Re\int_{\mathbb{R}^{2}}\partial_{t}\overline{Iu}(|Iu|^4Iu-I(|u|^4u))\,\text{d}x.
\end{align*}
Integrating this from $0$ to $\delta$ and using the Plancherel formula yields that
\begin{align*}
    E(Iu(\delta))-E(Iu(0))=&\begin{multlined}[t]
    \Re\int_{0}^{\delta}\int_{\sum_{i=1}^{6}\xi_{i}=0}\left(1-\frac{m_{23456}}{m_{2}m_{3}m_{4}m_{5}m_{6}}\right)\widehat{\partial_{t}\overline{Iu}}(\xi_{1})\widehat{Iu}(\xi_{2})\\\times\widehat{\overline{Iu}}(\xi_{3})\widehat{Iu}(\xi_{4})\widehat{\overline{Iu}}(\xi_{5})\widehat{Iu}(\xi_{6})\,\text{d}t
    \end{multlined}
    \\\lesssim&
    \begin{multlined}[t]
    \Bigg\vert\int_{0}^{\delta}\int_{\sum_{i=1}^{6}\xi_{i}=0}\left(1-\frac{m_{23456}}{m_{2}m_{3}m_{4}m_{5}m_{6}}\right)\widehat{\triangle^{2}\overline{Iu}}(\xi_{1})\widehat{Iu}(\xi_{2})\\\times\widehat{\overline{Iu}}(\xi_{3})\widehat{Iu}(\xi_{4})\widehat{\overline{Iu}}(\xi_{5})\widehat{Iu}(\xi_{6})\,\text{d}t\Bigg\vert
    \end{multlined}
    \\+&
    \begin{multlined}[t]
    \Bigg\vert\int_{0}^{\delta}\int_{\sum_{i=1}^{6}\xi_{i}=0}\left(1-\frac{m_{23456}}{m_{2}m_{3}m_{4}m_{5}m_{6}}\right)\widehat{\overline{I(\vert u\vert^{4}u)}}(\xi_{1})\widehat{Iu}(\xi_{2})\\\times\widehat{\overline{Iu}}(\xi_{3})\widehat{Iu}(\xi_{4})\widehat{Iu}(\xi_{5})\widehat{\overline{Iu}}(\xi_{6})\,\text{d}t\Bigg\vert
    \end{multlined}
    \\=:&\, Term_{1}+Term_{2}
\end{align*}
where, for simplicity, we have written $m_{j}=m(\xi_{j})$ and $m_{ij}=m(\xi_{ij})=m(\xi_{i}+\xi_{j})$ above. Next, we restrict our attention to showing that $Term_{1}+Term_{2}\lesssim N^{-k}$ for some $k>0$ depending on the spatial dimension, which is the content of the following proposition.
\begin{proposition}\label{enrgyincr_4}
Given $s>\frac{8}{7}$, $N\gg 1$, and initial data $u_{0}\in C_{c}^{\infty}(\mathbb{R}^{2})$ with $E(Iu_{0})\leq 1$, there exists a $\delta>0$ depending on the mass $\Vert u_{0}\Vert_{L^{2}(\mathbb{R}^{2})}$ of the initial data so that the solution $$u\in C([0,\delta],H^{s}(\mathbb{R}^{2}))$$ of \eqref{eq:4NLS} satisfies
\begin{equation}
    E(Iu(t))-E(Iu(0))=O(N^{-3+})
\end{equation}
for all $t\in[0,\delta]$.
\end{proposition}
Before getting to the proof of the Proposition \ref{enrgyincr_4}, let us see how the Proposition \ref{enrgyincr_4} proves the Theorem \ref{main_thm}. Substituting the scaled solution \eqref{eq:scalesymm} with large $\lambda$ in \eqref{eq:me<Hs} leads to
\begin{align}
    E(Iu_{0,\lambda})\lesssim& \big(\lambda^{-2s}N^{4-2s}+\lambda^{-4}\big)\big(1+\Vert u_{0}\Vert_{H^{s}(\mathbb{R}^{2})}\big)^{6}\label{eq:sc-me<Hs-1}\nonumber\\\lesssim& C_{0}\lambda^{-2s}N^{4-2s}\big(1+\Vert u_{0}\Vert_{H^{s}(\mathbb{R}^{2})}\big)^{6}.\nonumber
\end{align}
We pick the scaling parameter $\lambda=\lambda\big(N,\Vert u_{0}\Vert_{H^{s}(\mathbb{R}^{2})}\big)$ as 
\begin{equation}\label{eq:choice_scal_param}
    \lambda=\Big(\frac{1}{2C_{0}}\Big)^{\frac{1}{2s}}N^{\frac{4-2s}{2s}}\big(1+\Vert u_{0}\Vert_{H^{s}(\mathbb{R}^{2})}\big)^{\frac{6}{2s}}
\end{equation}
so that $E(Iu_{0,\lambda})\leq\frac{1}{2}$ and then we apply Proposition \ref{enrgyincr_4} to the scaled initial data $u_{0,\lambda}$ iteratively, until the size of $E(Iu_{\lambda}(t))$ reaches $1$. To be more precise, we run Proposition \ref{enrgyincr_4} at least $C_{1}N^{3-}$ many times to achieve
\begin{equation}\label{eq:me_sim_1}
    E(Iu_{\lambda}(C_{1}N^{3-}\delta))\sim 1.
\end{equation}
Besides, for any time parameter $T_{0}\gg 1$, we choose $N\gg 1$ so that
\begin{equation}\label{choice_of_N}
    T_{0}\sim\frac{N^{3-}}{\lambda^{4}}C_{1}\delta\sim N^{\frac{7s-8}{s}-}.
\end{equation}
Therefore, combining \eqref{eq:choice_scal_param}, \eqref{eq:me_sim_1} and \eqref{choice_of_N}, it turns out that
\begin{equation}
    E(Iu(T_{0}))=\lambda^{4}E(Iu_{\lambda}(\lambda^{4}T_{0}))\lesssim_{\delta,\Vert u_{0}\Vert_{H^{s}(\mathbb{R}^{2})}}\lambda^{4}\sim N^{\frac{8-4s}{s}}\sim T_{0}^{\frac{8-4s}{7s-8}+}\label{eq:me_large_time}.
\end{equation}
As a result, using \eqref{eq:Hs<me} and \eqref{eq:me_large_time}, we conclude that
\begin{equation}
    \Vert u(T_{0})\Vert_{H^{s}(\mathbb{R}^{n})}\lesssim T_{0}^{\frac{8-4s}{7s-8}+}.\nonumber
\end{equation}
The desired polynomial bound \eqref{eq:polbound} is then obtained when $s>\frac{8}{7}$.
\section{Proof of Proposition \ref{enrgyincr_4}}
Before giving the proof of the Proposition \ref{enrgyincr_4}, we need to make some preparation concerning the terms $Term_{j}$, $j=1,2$. We start estimating $Term_{1}$. By means of the Littlewood-Paley decomposition, we define
\begin{equation*}
    \widehat{u_{N_{1}}}= \widehat{P_{N_{1}}\triangle^{2}Iu},\quad \widehat{u_{N_{i}}}=\widehat{P_{N_{i}}Iu},\,i=2,3,4,5,6
\end{equation*}
where $P_{N}$ is the Littlewood-Paley projection operator \eqref{eq:L-P_opt} and $N_{j}=2^{k_{j}}$, $k_{j}\in\{0,1,2,\dots\}$ for $j=1,\dots,6$. Via the decomposition, we have
\begin{align}
    Term_{1}
    \leq&\begin{multlined}[t]\label{eq:L-P_piece_T_1}
    \sum_{N_{1},\dots,N_{6}}\Bigg\vert\int_{0}^{\delta}\int_{\sum_{i=1}^{6}\xi_{i}=0}\Big(1-\frac{m_{23456}}{m_{2}m_{3}m_{4}m_{5}m_{6}}\Big)\widehat{\overline{u_{N_{1}}}}(\xi_{1})\widehat{u_{N_{2}}}(\xi_{2})\widehat{\overline{u_{N_{3}}}}(\xi_{3})\\\times\widehat{u_{N_{4}}}(\xi_{4}) \widehat{\overline{u_{N_{5}}}}(\xi_{5})\widehat{u_{N_{6}}}(\xi_{6})\Bigg\vert
    \end{multlined}
\end{align}
where the sums are taken over the dyadic numbers $N_{i}=2^{k_{i}}$, $k_{i}\in\{0,1,2,\dots\}$ and $\langle\xi_{i}\rangle\sim N_{i}$ for $i=1,\dots,6$. Due to the symmetry of the variables $\xi_{2},\xi_{3},\xi_{4},\xi_{5},\xi_{6}$ in the multiplier, we may restrict our attention to the case $N_{2}\geq N_{3}\geq N_{4}\geq N_{5}\geq N_{6}$ only. Then we always have $N_{1}\lesssim N_{2}$. Henceforth, the strategy is to bound each integral in the sum \eqref{eq:L-P_piece_T_1} depending on the relative size of the frequencies, and then to sum all the bounds via \eqref{eq:L-P_ineq}. Also, without loss of generality, we may assume that the spatial Fourier transform of dyadic pieces are non-negative. Under this assumption, we can take the multiplier out with a pointwise bound
\begin{equation*}
    \Big| 1-\frac{m_{23456}}{m_{2}m_{3}m_{4}m_{5}m_{6}}\Big|\lesssim C(N_{1},N_{2},N_{3},N_{4},N_{5},N_{6})
\end{equation*}
where $C(N_{1},N_{2},N_{3},N_{4},N_{5},N_{6})=: C>0$ will be determined suitably in the different frequency interaction cases. Combining the arguments above, we wish to demonstrate
\begin{align}\label{eq: main_T_1}
    C\Big\Vert \prod_{j=1}^6u_{N_{j}}\Big\Vert_{L^{1}_{t,x}}\lesssim N^{-3+}N_{2}^{0-}\Vert u_{N_{1}}\Vert_{X^{-2,\frac{1}{2}+}_{\delta}}\prod_{j=2}^{6}\Vert u_{N_{j}}\Vert_{X^{2,\frac{1}{2}+}_{\delta}}.
\end{align}
Note that all $L^{p}_{t}L^{q}_{x}$ and $X^{s,b}$ norms in \eqref{eq: main_T_1} are taken on the domain $[0,\delta]\times\mathbb{R}^{2}$ and we shall keep this notation once we start the proof of Proposition \ref{enrgyincr_4}. Handling the $Term_{2}$ similar to the $Term_1$, we might express the $Term_{2}$ as follows
\begin{multline}\label{eq:Term_2}
    \Bigg|\sum_{N_{6}\geq\dots\geq N_{10}}\int_{0}^{\delta}\int_{\sum_{i=1}^{10}\xi_{i}=0}\Big(1-\frac{m_{678910}}{m_{6}m_{7}m_{8}m_{9}m_{10}}\Big)P_{N_{12345}}\widehat{\overline{I(|u|^4u)}}(\xi_{12345})\\\times\widehat{Iu_{N_{6}}}(\xi_{6})\widehat{\overline{Iu_{N_{7}}}}(\xi_{7})\widehat{Iu_{N_{8}}}(\xi_{8})\widehat{\overline{Iu_{N_{9}}}}(\xi_{9})\widehat{Iu_{N_{10}}}(\xi_{10})\Bigg|
\end{multline}
where $P_{N_{12345}}$ is the Littlewood-Paley projection operator onto the dyadic shell $N_{12345}\sim\langle\xi_{12345}\rangle$. The dyadic sum is given as in \eqref{eq:Term_2} due to again the symmetry of the multiplier. 
Therefore, it is sufficient to show that
\begin{multline}\label{eq:main_est_T2}
\Bigg|\int_{\sum_{i=1}^{10}\xi_{i}=0}\Big(1-\frac{m_{678910}}{m_{6}m_{7}m_{8}m_{9}m_{10}}\Big)P_{N_{12345}}\widehat{\overline{I(|u|^4u)}}(\xi_{12345})\\\times\widehat{Iu_{N_{6}}}(\xi_{6})\widehat{\overline{Iu_{N_{7}}}}(\xi_{7})\widehat{Iu_{N_{8}}}(\xi_{8})\widehat{\overline{Iu_{N_{9}}}}(\xi_{9})\widehat{Iu_{N_{10}}}(\xi_{10})\Bigg|\lesssim N^{-3+}N_{6}^{0-}\Vert Iu\Vert_{X^{2,\frac{1}{2}+}_{\delta}}^{5}\prod_{j=6}^{10}\Vert Iu_{N_{j}}\Vert_{X^{2,\frac{1}{2}+}_{\delta}}.
\end{multline}
We may assume $N_{12345}\lesssim N_{6}$ and $N_{6}\gtrsim N$ to omit the case where the multiplier inside is zero. The decay factor $N_{6}^{0-}$ allows us to sum on the dyadic numbers $N_{12345},N_{6},N_{7},N_{8},N_{9},N_{10}$ and we do not have to decompose separately the terms of $I(|u|^4u)$. We again take the symbol out of the integral with a pointwise bound for each frequency interaction case as
\begin{equation*}
    \Big| 1-\frac{m_{678910}}{m_{6}m_{7}m_{8}m_{9}m_{10}}\Big|\lesssim\frac{m(N_{12345})}{m(N_{6})m(N_{7})m(N_{8})m(N_{9})m(N_{10})}
\end{equation*}
and undo the Plancherel formula by assuming that the Fourier transform of the dyadic pieces of $u$ are non-negative. In addition, we need the following result:
\begin{lemma}[Modified local existence] Given $\frac{8}{7}<s<2$ and the initial data $u_{0}$ for the equation \eqref{eq:4NLS} with $E(Iu_{0})\leq 1$, there is a constant $\delta=\delta(\Vert u_{0}\Vert_{L^{2}(\mathbb{R}^{2})})>0$ such that  on $[0,\delta]$ the solution $u$ satisfies
\begin{equation*}
    \Vert Iu\Vert_{X^{2,\frac{1}{2}+}_{\delta}}\lesssim 1.
\end{equation*}
\end{lemma}
\begin{proof}
    We implement the standard iteration argument to prove the local existence of the modified equation
    \begin{equation}\label{eq:M4NLS}
        \begin{cases}   
            i\partial_{t}Iu-\triangle^{2}Iu=I(\vert u\vert^{4}u),\\
            Iu(0,x)=Iu_{0}(x).
        \end{cases}
    \end{equation}
    We require the following $X^{s,b}$ estimates, see e.g. \cite{tao2006nonlinear},
    \begin{align}
    \Vert e^{it\triangle^{2}}u\Vert_{X^{2,\frac{1}{2}+}_{\delta}}\lesssim&\,\Vert u\Vert_{H^{2}_{x}},\label{eq:typ_est_1}\\
    \Big\Vert\int_{0}^{t}e^{i(t-t')\triangle^{2}}U(x,t')\,\text{d}t'\Big\Vert_{X^{2,\frac{1}{2}+}}\lesssim&\,\Vert U\Vert_{X^{2,-\frac{1}{2}+}_{\delta}},\label{eq:typ_est_2}\\
    \Vert U\Vert_{X^{2,-b}_{\delta}}\lesssim&\,\delta^{b-b'}\Vert U\Vert_{X^{2,-b'}}\label{eq:typ_est_3}
    \end{align}
    where $0<b'<b<\frac{1}{2}$.  By Duhamel's formula, we write
\begin{equation}\label{eq:duhamel}
    Iu(t,x)=e^{it\triangle^{2}}Iu_{0}(x)+i\int_{0}^{t}e^{i(t-t')\triangle^{2}}I(\vert u\vert^{4}u)(t',x)\,\text{d}t'
\end{equation}
for $t\in[0,\delta]$. Using \eqref{eq:typ_est_1}--\eqref{eq:duhamel}, we see that
\begin{align*}
    \Vert Iu\Vert_{X^{2,\frac{1}{2}+}_{\delta}}\lesssim \Vert Iu_{0}\Vert_{H^{2}}+\delta^{0+}\Vert I(|u|^4u)\Vert_{X^{2,-\frac{1}{2}++}_{\delta}}.
\end{align*}
Note that
\begin{equation}\label{eq:H2<EI}
    \Vert Iu_{0}\Vert_{H^{2}_{x}}\lesssim (E(Iu_{0}))^{1/2}+\Vert u_{0}\Vert_{L^{2}_{x}}\lesssim 1+\Vert u_{0}\Vert_{L^{2}_{x}}.
\end{equation}
If we show that
\begin{equation}\label{eq:mult_lin_est}
    \Vert I(|u|^4u)\Vert_{X^{2,-\frac{1}{2}++}_{\delta}}\lesssim\Vert Iu\Vert^{5}_{X^{2,\frac{1}{2}+}_{\delta}},
\end{equation}
then we are done because combining \eqref{eq:H2<EI} with \eqref{eq:mult_lin_est} entails that
\begin{equation}\label{eq:mass_dpn_delta}
    \Vert Iu\Vert_{X_{\delta}^{2,\frac{1}{2}+}}\lesssim 1+\Vert u_{0}\Vert_{L^{2}_{t,x}}+\delta^{0+}\Vert Iu\Vert_{X_{\delta}^{2,\frac{1}{2}+}}^{5}
\end{equation}
and then running a bootstrap or continuity argument (see chapter $1$ of \cite{tao2006nonlinear}) to \eqref{eq:mass_dpn_delta}, we obtain the desired result. Note that the $\Vert u_{0}\Vert_{L^{2}_{x}}$ dependence of $\delta$ can be seen in \eqref{eq:mass_dpn_delta}. We also note that $\delta$ must be sufficiently small in order to apply the bootstrap argument. Using the interpolation lemma in \cite{Colliander2004}, it remains to demonstrate that
\begin{equation*}
     \Vert |u|^4u\Vert_{X^{s,-\frac{1}{2}++}_{\delta}}\lesssim\Vert u\Vert_{X^{s,\frac{1}{2}+}_{\delta}}^{5}
\end{equation*}
for $\frac{8}{7}<s<2$. By the fractional Leibniz rule, this boils down to showing 
\begin{equation*}
    \Vert (\langle\nabla\rangle^{s} u)|u|^4\Vert_{X^{0,-\frac{1}{2}++}_{\delta}}\lesssim\Vert u\Vert_{X^{s,\frac{1}{2}+}_{\delta}}^{5}
\end{equation*}
which is equivalent by duality to showing that 
\begin{equation}\label{eq:duality}
    \Bigg|\int_{0}^{\delta}\int_{\mathbb{R}^{2}}(\langle\nabla\rangle^{s}u)|u|^4f\,\text{d}x\text{d}t\Bigg|\lesssim \Vert u\Vert_{X^{s,\frac{1}{2}+}_{\delta}}^{5}.
\end{equation}
Applying Hölder's inequality, we have
\begin{equation*}
    \text{LHS of}\,\eqref{eq:duality} \leq\Vert\langle\nabla\rangle^{s}u\Vert_{L^{6}_{t,x}}\Vert u\Vert_{L^{6}_{t,x}}^{3}\Vert u\Vert_{L^{6+}_{t,x}}\Vert f\Vert_{L^{6-}_{t,x}}.
\end{equation*}
We then apply $L^{6}_{t,x}$-Strichartz estimate to the first two factors to get
\begin{equation*}
    \Vert\langle\nabla\rangle^{s}u\Vert_{L^{6}_{t,x}}\lesssim\Vert u\Vert_{X^{s,\frac{1}{2}+}_{\delta}}
\end{equation*}
and
\begin{equation*}
    \Vert u\Vert_{L^{6}_{t,x}}^{3}\lesssim\Vert u\Vert_{X^{0,\frac{1}{2}+}_{\delta}}^{3}\lesssim\Vert u\Vert_{X^{s,\frac{1}{2}+}_{\delta}}^{3}.
\end{equation*}
For the third factor, we first apply Sobolev embedding and then $L^{6+}_{t}L^{6-}_{x}$-Strichartz estimate to get
\begin{equation*}
    \Vert u\Vert_{L^{6+}_{t,x}}\lesssim\Vert\langle\nabla\rangle^{0+} u\Vert_{L^{6+}_{t}L^{6-}_{x}}\lesssim\Vert u\Vert_{X^{0+,\frac{1}{2}+}_{\delta}}\lesssim\Vert u\Vert_{X^{s,\frac{1}{2}+}_{\delta}}.
\end{equation*}
For the last factor, note that
\begin{equation*}
    \Vert f\Vert_{L^{2}_{t,x}}\lesssim\Vert f\Vert_{X^{0,0}_{\delta}}
\end{equation*}
and
\begin{equation*}
    \Vert f\Vert_{L^{6}_{t,x}}\lesssim\Vert f\Vert_{X^{0,\frac{1}{2}+}_{\delta}}.
\end{equation*}
Interpolating the estimates above, we obtain
\begin{equation*}
    \Vert f\Vert_{L^{6-}_{t,x}}\lesssim\Vert f\Vert_{X^{0,\frac{1}{2}--}_{\delta}}.
\end{equation*}
Combining all the estimates and taking the supremum of \eqref{eq:duality} over $\Vert f\Vert_{X^{0,\frac{1}{2}--}_{\delta}}=1$ completes the proof.
\end{proof}
\begin{proof}[Proof of Proposition \ref{enrgyincr_4}]
By the previous discussion, it is sufficient to establish \eqref{eq: main_T_1} and \eqref{eq:main_est_T2}. We start with $Term_{1}$. There are two frequency interaction cases. We omit the case $N\gg N_{2}$, as in this case the multiplier in the integral in \eqref{eq:L-P_piece_T_1} vanishes.

\noindent
\textbf{Case 1: }$N_{1}\sim N_{2}\gtrsim N\gg N_{3}$. \\ In this region, the multiplier is estimated by
\begin{equation*}
    \Big|1-\frac{m_{23456}}{m_{2}m_{3}m_{4}m_{5}m_{6}}\Big|\lesssim\frac{N_{3}}{N_{2}}.
\end{equation*} Using this bound and Hölder's inequality, we get
\begin{equation*}
    \text{LHS of \eqref{eq: main_T_1}}\lesssim\frac{N_{3}}{N_{2}}\Vert u_{N_{1}}\Vert_{L^{2}_{t}L^{\infty}_{x}}\Vert u_{N_{2}}\Vert_{L^{2}_{t}L^{\infty}_{x}}\prod_{j=3}^{6}\Vert u_{N_{j}}\Vert_{L^{\infty}_{t}L^{4}_{x}}.
\end{equation*}
   For high frequency components, from \eqref{keyestimate} with $\mu=1$, and for low frequency components, from \eqref{keyestimate} with $\mu=0$ and the Sobolev embedding, we obtain
\begin{align*}
    \Vert u_{N_{j}}\Vert_{L^{2}_{t}L^{\infty}_{x}}&\lesssim N_{1}^{-1}\Vert u_{N_{j}}\Vert_{X^{0,\frac{1}{2}+}_{\delta}},\quad j=1,2,\\
    \Vert u_{N_{j}}\Vert_{L^{\infty}_{t}L^{4}_{x}}&\lesssim\Vert u_{N_{j}}\Vert_{L^{\infty}_{t}\Dot{H}^{\frac{1}{2}}_{x}}\lesssim N_{j}^{1/2}\Vert u_{N_{j}}\Vert_{X^{0,\frac{1}{2}+}_{\delta}},\quad j=3,4,5,6.
\end{align*}
Collecting these estimates leads to
\begin{equation*}
    \text{LHS of \eqref{eq: main_T_1}}\lesssim\frac{N_{3}}{N_{2}}\frac{1}{N_{1}N_{2}}\Vert u_{N_{1}}\Vert_{X^{0,\frac{1}{2}+}_{\delta}}\Vert u_{N_{2}}\Vert_{X^{0,\frac{1}{2}+}_{\delta}}\prod_{j=3}^{6}N_{j}^{1/2}\Vert u_{N_{j}}\Vert_{X^{0,\frac{1}{2}+}_{\delta}}.
\end{equation*}
In order to establish the inequality \eqref{eq: main_T_1}, using \eqref{eq:Berns_ineq}, it is therefore enough to see that
\begin{equation*}
    \frac{N_{3}}{N_{2}}\frac{1}{N_{1}N_{2}}\frac{N_{1}^{2}}{N_{2}^{2}N_{3}^{2}N_{4}^{2}N_{5}^{2}N_{6}^{2}}N_{3}^{1/2}N_{4}^{1/2}N_{5}^{1/2}N_{6}^{1/2}\lesssim N^{-3+}N_{2}^{0-}
\end{equation*}
which is true in Case $1$.

\noindent
\textbf{Case 2: }$N_{2}\sim N_{3}\gtrsim N$ and $N_{1}\lesssim N_{2}$.\\ There are two sub-cases to consider:

\noindent
\textbf{Sub-case 2a.} $N_{1}\sim N_{2}\gg N_{3}\gtrsim N$. In this case, we use the following trivial bound for the multiplier
\begin{equation*}
    \Big|1-\frac{m_{23456}}{m_{2}m_{3}m_{4}m_{5}m_{6}}\Big|\lesssim\frac{m(N_{1})}{m(N_{2})m(N_{3})m(N_{4})m(N_{5})m(N_{6})}.
\end{equation*}
Again using Hölder's inequality, we get
\begin{align*}
    \text{Left side of \eqref{eq: main_T_1}}\lesssim\frac{m(N_{1})}{m(N_{2})m(N_{3})m(N_{4})m(N_{5})m(N_{6})}\Vert u_{N_{1}}\Vert_{L^{2}_{t}L^{\infty}_{x}}\Vert u_{N_{2}}\Vert_{L^{2}_{t}L^{\infty}_{x}}\prod_{j=3}^{6}\Vert u_{N_{j}}\Vert_{L^{\infty}_{t}L^{4}_{x}}.
\end{align*}
Applying \ref{keyestimate} in a similar vein as in Case 1 along with the Sobolev embedding, we obtain
\begin{align*}
    \Vert u_{N_{j}}\Vert_{L^{2}_{t}L^{\infty}_{x}}&\lesssim N_{j}^{-1}\Vert u_{N_{j}}\Vert_{X^{0,\frac{1}{2}+}_{\delta}},\quad j=1,2,\\
    \Vert u_{N_{j}}\Vert_{L^{\infty}_{t}L^{4}_{x}}&\lesssim\Vert u_{N_{1}}\Vert_{L^{\infty}_{t}\Dot{H}_{x}^{\frac{1}{2}}}\lesssim N_{j}^{1/2}\Vert u_{N_{j}}\Vert_{X^{0,\frac{1}{2}+}_{\delta}},\quad j=3,4,5,6.
\end{align*}
Therefore, we get
\begin{equation*}
    \text{Left side of \eqref{eq: main_T_1}}\lesssim\frac{m(N_{1})}{m(N_{2})m(N_{3})m(N_{4})m(N_{5})m(N_{6})}\frac{1}{N_{2}^{2}N_{3}^{3/2}N_{4}^{3/2}N_{5}^{3/2}N_{6}^{3/2}}\Vert u_{N_{1}}\Vert_{X^{-2,\frac{1}{2}+}_{\delta}}\prod_{j=2}^{6}\Vert u_{N_{j}}\Vert_{X^{2,\frac{1}{2}+}_{\delta}}.
\end{equation*}
Hence, it suffices to show
\begin{equation}\label{eq:eq_1_term_1}
    \frac{N_{2}^{-2+}N^{3-}}{m(N_{3})N_{3}^{3/2}m(N_{4})N_{4}^{3/2}m(N_{5})N_{5}^{3/2}m(N_{6})N_{6}^{3/2}}\lesssim 1.
\end{equation}
As $m(x)x^{3/2}$ is an increasing function for $x\geq 0$, we have $m(N_{3})N_{3}^{3/2}\gtrsim N^{3/2}$ and $m(N_{j})N_{j}^{3/2}\gtrsim 1$ for $j=4,5,6$. Plugging these bounds into the left side of \eqref{eq:eq_1_term_1} gives us
\begin{equation*}
   \text{LHS of \eqref{eq:eq_1_term_1}}\lesssim N_{2}^{-2+}N^{\frac{3}{2}-}\lesssim 1
\end{equation*}
since $N_{2}\gtrsim N$.

\noindent
\textbf{Sub-case 2b. }$N_{2}\sim N_{3}\gtrsim N$ and $N_{1}\lesssim N_{2}$.\\ We can bound the multiplier as
\begin{equation*}
    \Big|1-\frac{m_{23456}}{m_{2}m_{3}m_{4}m_{5}m_{6}}\Big|\lesssim\frac{m(N_{1})}{m(N_{2})m(N_{3})m(N_{4})m(N_{5})m(N_{6})}.
\end{equation*}
By Hölder's inequality, we obtain
\begin{align*}
    \text{Left side of \eqref{eq: main_T_1}}\lesssim \frac{m(N_{1})}{m(N_{2})m(N_{3})m(N_{4})m(N_{5})m(N_{6})}\Vert u_{N_{1}}\Vert_{L^{\infty}_{t}L^{2}_{x}}\Vert u_{N_{2}}\Vert_{L^{2}_{t}L^{\infty}_{x}}\Vert u_{N_{3}}\Vert_{L^{2}_{t}L^{\infty}_{x}}\prod_{j=4}^{6}\Vert u_{N_{j}}\Vert_{L^{\infty}_{t}L^{6}_{x}}.
\end{align*}
Utilizing Sobolev embedding, the estimate \eqref{keyestimate} with $\mu=1$ for $u_{N_{2}},u_{N_{3}}$ and with $\mu=0$ for the remaining factors, and finally using \eqref{eq:Berns_ineq}, we conclude that
\begin{equation*}
    \text{Left side of \eqref{eq: main_T_1}}\lesssim\frac{m(N_{1})N_{1}^{2}}{(m(N_{2}))^{2}N_{2}^{6}m(N_{4})N_{4}^{4/3}m(N_{5})N_{5}^{4/3}m(N_{6})N_{6}^{4/3}}\Vert u_{N_{1}}\Vert_{X^{-2,\frac{1}{2}+}_{\delta}}\prod_{j=2}^{6}\Vert u_{N_{j}}\Vert_{X^{2,\frac{1}{2}+}_{\delta}}.
\end{equation*}
Thus, it suffices to show
\begin{equation*}
    \frac{m(N_{1})N_{1}^{2}N^{3-}N_{2}^{0+}}{(m(N_{2}))^{2}N_{2}^{6}m(N_{4})N_{4}^{4/3}m(N_{5})N_{5}^{4/3}m(N_{6})N_{6}^{4/3}}\lesssim 1.
\end{equation*}
As $m(x)x^{p}$ is increasing for $x\geq 0$ and $p>6/7$, we have $m(N_{1})N_{1}^{2}\lesssim m(N_{2})N_{2}^{2}$, $m(N_{2})N_{2}^{4-}\gtrsim N^{4-}$, and $m(N_{j})N_{j}^{4/3}\gtrsim 1$ for $j=4,5,6$. Using these bounds, it is enough to establish
\begin{equation*}
    \frac{N^{3-}N_{2}^{0+}}{m(N_{2})N_{2}^{4-}N_{2}^{0+}}\lesssim N^{3-}N^{-4+}=N^{-1}\lesssim 1
\end{equation*}
since $N\gg 1$. This completes the assertion in \eqref{eq: main_T_1}.

It remains to prove that $Term_{2}\lesssim N^{-3+}$. We start by estimating the left side of \eqref{eq:main_est_T2}. Applying Hölder's inequality, we get
\begin{multline}\label{eq:term_2_holder}
    \text{LHS of \eqref{eq:main_est_T2}}\lesssim\frac{m(N_{12345})}{m(N_{6})m(N_{7})m(N_{8})m(N_{9})m(N_{10})}\Vert P_{N_{12345}}I(\vert u\vert^{4}u)\Vert_{L^{2}_{t,x}}\\\times\Vert Iu_{N_{6}}\Vert_{L^{6}_{t,x}}\Vert Iu_{N_{7}}\Vert_{L^{6}_{t,x}}\Vert Iu_{N_{8}}\Vert_{L^{6}_{t,x}}\Vert Iu_{N_{9}}\Vert_{L^{\infty}_{t,x}}\Vert Iu_{N_{10}}\Vert_{L^{\infty}_{t,x}}.
\end{multline}
We need the following lemma to proceed further.
\begin{lemma}\label{lemma_T2}
    Let $u,u_{N_{6}},\dots,u_{N_{10}}$ be given as in \eqref{eq:Term_2}. Then we have
    \begin{align}
        \Vert P_{N_{12345}}I(\vert u\vert^{4}u)\Vert_{L^{2}_{t,x}}&\lesssim\langle N_{12345}\rangle^{-2}\Vert Iu\Vert_{X^{2,\frac{1}{2}+}_{\delta}}^{5}\label{eq:eq1}\\
        \Vert Iu_{N_{j}}\Vert_{L^{6}_{t,x}}&\lesssim\langle N_{j}\rangle^{-2}\Vert Iu_{N_{j}}\Vert_{X^{2,\frac{1}{2}+}_{\delta}},\quad j=6,7,8,\label{eq:eq2}\\
        \Vert Iu_{N_{j}}\Vert_{L^{\infty}_{t,x}}&\lesssim\langle N_{j}\rangle^{-1}\Vert Iu_{N_{j}}\Vert_{X^{2,\frac{1}{2}+}_{\delta}},\quad j=9,10.\label{eq:eq3}
    \end{align}
\end{lemma}
\begin{proof}
     The estimate \eqref{eq:eq1} amounts to showing that
    \begin{equation*}
    \Vert\langle\nabla\rangle^{2}P_{N_{12345}}I(|u|^4u)\Vert_{L^{2}_{t,x}}\lesssim\Vert Iu\Vert_{X^{2,\frac{1}{2}+}_{\delta}}^{5}. 
\end{equation*}
The pseudo-differential operator $\langle\nabla\rangle^{2}I$ is of positive order $s>\frac{8}{7}$. So, it obeys the fractional Leibniz rule which implies that we may work on a typical term
\begin{equation*}
    \Vert P_{N_{12345}}(\langle\nabla\rangle^{2}Iu)|u|^4\Vert_{L^{2}_{t,x}}.
\end{equation*}
Applying Hölder's inequality, we get
\begin{equation*}
    \Vert P_{N_{12345}}(\langle\nabla\rangle^{2}Iu)|u|^4\Vert_{L^{2}_{t,x}}\lesssim\Vert(\langle\nabla\rangle^{2} Iu)|u|^4\Vert_{L^{2}_{t,x}}\leq\Vert\langle\nabla\rangle^{2} Iu\Vert_{L^{4}_{t}L^{\infty}_{x}}\Vert u\Vert_{L^{16}_{t}L^{8}_{x}}^{4}.
\end{equation*}
We shall utilize the estimate \eqref{keyestimate} with $\mu=0$ for all factors above to get
\begin{equation*}
    \Vert\langle\nabla\rangle^{2} Iu\Vert_{L^{4}_{t}L^{\infty}_{x}}\lesssim\Vert Iu\Vert_{X^{s,\frac{1}{2}+}_{\delta}}\lesssim\Vert Iu\Vert_{X^{2,\frac{1}{2}+}_{\delta}},
\end{equation*}
and also by Sobolev embedding, 
\begin{equation*}
    \Vert u\Vert_{L^{16}_{t}L^{8}_{x}}\lesssim\Vert\langle\nabla\rangle^{1/2} u\Vert_{L^{16}_{t}L^{\frac{8}{3}}_{x}}\lesssim\Vert \langle\nabla\rangle^{1/2}u\Vert_{X^{0,\frac{1}{2}+}_{\delta}}\lesssim\Vert Iu\Vert_{X^{2,\frac{1}{2}+}_{\delta}}
\end{equation*}
which completes the proof of \eqref{eq:eq1}. The inequality \eqref{eq:eq2} follows directly if we apply estimate \ref{keyestimate} to the biharmonic admissible pair $(6,6)$ 
\begin{equation*}
    \Vert Iu_{N_{j}}\Vert_{L^{6}_{t,x}}\lesssim\Vert Iu_{N_{j}}\Vert_{X^{0,\frac{1}{2}+}_{\delta}}\lesssim\langle N_{j}\rangle^{-2}\Vert Iu_{N_{j}}\Vert_{X^{2,\frac{1}{2}+}_{\delta}},\quad j=6,7,8.
\end{equation*}
Lastly, \eqref{eq:eq3} is established by the Fourier inversion formula and Cauchy-Schwarz inequality
\begin{align*}
    \vert Iu(x)\vert\leq\int_{\vert\xi\vert\sim N_{j}}\langle\xi\rangle^{2}\langle\xi\rangle^{-2}\vert\widehat{Iu_{N_{j}}}(\xi)\vert \,\text{d}\xi&\lesssim\langle N_{j}\rangle^{-1}\Bigg(\int_{\vert\xi\vert\sim N_{j}}\langle\xi\rangle^{4}\vert\widehat{Iu_{N_{j}}}(\xi)\vert^{2} \text{d}\xi\Bigg)^{1/2}\\
    &=\langle N_{j}\rangle^{-1}\Vert Iu_{N_{j}}(t)\Vert_{H^{2}_{x}},\quad j=9,10.
\end{align*}
This implies that
\begin{equation*}
    \Vert Iu_{N_{j}}\Vert_{L^{\infty}_{t,x}}\lesssim\langle N_{j}\rangle^{-1}\Vert Iu_{N_{j}}\Vert_{L^{\infty}_{t}H^{2}_{x}}\lesssim\langle N_{j}\rangle^{-1}\Vert Iu_{N_{j}}\Vert_{X^{2,\frac{1}{2}+}_{\delta}},\quad j=9,10.
\end{equation*}
\end{proof}
\noindent Applying Lemma \ref{lemma_T2}, we obtain
\begin{multline*}
    \text{RHS of \eqref{eq:term_2_holder}}\lesssim\frac{m(N_{12345})}{m(N_{6})m(N_{7})m(N_{8})m(N_{9})m(N_{10})}\\\times\frac{1}{\langle N_{12345}\rangle^{2}\langle N_{6}\rangle^{2}\langle N_{7}\rangle^{2}\langle N_{8}\rangle\langle N_{9}\rangle\langle N_{10}\rangle}\Vert Iu\Vert_{X^{2,\frac{1}{2}+}_{\delta}}^{5}\prod_{j=6}^{10}\Vert Iu_{N_{j}}\Vert_{X^{2,\frac{1}{2}+}_{\delta}}.
\end{multline*}
Thus, it suffices to show
\begin{equation}\label{eq:last_dyad_est_T2}
    \frac{m(N_{12345})}{m(N_{6})m(N_{7})m(N_{8})m(N_{9})m(N_{10})}\frac{N^{3-}N_{6}^{0+}}{\langle N_{12345}\rangle^{2}\langle N_{6}\rangle^{2}\langle N_{7}\rangle^{2}\langle N_{8}\rangle^{2}\langle N_{9}\rangle\langle N_{10}\rangle}\lesssim 1.
\end{equation}
There are two frequency interaction cases to consider

\noindent
\textbf{Case 1: }$N_{12345}\sim N_{6}\gtrsim N\gg N_{7}$ 
\\In this case,
\begin{equation*}
    \text{LHS of \eqref{eq:last_dyad_est_T2}}\lesssim N^{-3+}N_{6}^{-4+}\lesssim 1
\end{equation*}
since the function $m(x)x$ is increasing for $x\geq 0$ when $s\geq 1$.

\noindent
\textbf{Case 2: }$N_{6}\geq N_{7}\gtrsim N$ and $N_{12345}\lesssim N_{6}$. \\ There are two sub-cases to handle:

\noindent
\textbf{Sub-case 2a. }$N_{6}\sim N_{7}\gtrsim N$ and $N_{12345}\lesssim N_{6}$.

\noindent
In this case
\begin{equation*}
    \text{LHS of \eqref{eq:last_dyad_est_T2}}\lesssim\frac{N^{3-}N_{6}^{0+}m(N_{12345})}{\langle N_{12345}\rangle^{2}(m(N_{6}))^{2}N^{4-}_{6}N_{6}^{0+}}\lesssim N^{3-}N^{-4+}\lesssim 1
\end{equation*}
since for $s>8/7$ the function $m(x)x^{p}$ is increasing when $p>6/7$, and  $m(N_{12345})\langle N_{12345}\rangle^{-2}\lesssim 1$.

\noindent
\textbf{Sub-case 2b. }$N_{12345}\sim N_{6}\gg N_{7}\gtrsim N$.\\ Under the conditions, we have
\begin{equation*}
    \text{LHS \eqref{eq:last_dyad_est_T2}}\lesssim\frac{N^{3-}N_{6}^{0+}}{N_{6}^{4}m(N_{7})N_{7}^{2}m(N_{8})N_{8}^{2}m(N_{9})N_{9}^{2}m(N_{10})N_{10}^{2}}\lesssim N^{3-}N_{6}^{-4+}\lesssim 1
\end{equation*}
as $m(x)x^{2}$ is increasing, i.e., $m(N_{j})N_{j}^{2}\gtrsim 1$ for $j=7,8,9,10$. Thus, we conclude that $Term_{2}\lesssim N^{-3+}$ which completes the proof.
\end{proof}
\noindent \textbf{Remark. }Without using the refined version of Strichartz estimate in Lemma \ref{X_s,b-str}, i.e., if we are not able to take $\mu>0$, then we would only control the growth of the modified energy with $N^{-1+}$. Note that the left side of \eqref{eq: main_T_1} can be bounded by $N^{-\frac{7}{2}+}$ and $N^{-4+}$ in the Sub-case $2a$ and the Sub-case $2b$, respectively. Nevertheless in the Case $1$, the bound $N^{-3+}$ cannot be improved to a smaller exponent of $N$. In addition, notice that $Term_{2}$ may also be better controlled with $N^{-4+}$ in all cases.

\nocite{*}
\bibliographystyle{abbrv}
\bibliography{reference.bib}

\end{document}